\documentclass[reqno,fleqn,11pt]{amsart}

\usepackage{amsmath}
\usepackage{amssymb}
\usepackage{amsfonts}

\usepackage{amsthm}
\usepackage{enumerate}
\usepackage{color}

\newcommand{\R}{\mathbb{R}}
\newcommand{\C}{\mathbb{C}}

\newcommand{\f}{\rightarrow}
\newcommand{\deb}{\bar\partial}
\newcommand{\hilb}{\mathcal{H}}
\newcommand{\de}{\partial}
\newcommand{\K}{K\"{a}hler}

\newcommand{\lmb}{\lambda}
\newcommand{\g}{h}
\newcommand{\re}{\mathop{\mathrm{Re}}}

\newcommand{\Aut}{\operatorname{Aut}}
\newcommand{\Isom}{\operatorname{Isom}}

\newtheorem{thm}{Theorem}
\newtheorem{prop}{Proposition}
\newtheorem{lem}[prop]{Lemma}

\newtheorem{exmp}[prop]{Example}
\newtheorem{rmk}[prop]{Remark}

\newcommand{\A}{\mathcal A}

\newcommand{\ov}[1]{\overline{#1}}

\newcommand{\ol}{\mathrm{Hol}}

\newcommand{\pf}{\textit{Proof. }}

\newcommand{\set}[2]{ \left\{\,#1\,;\,#2\,\right\} }

\newcommand{\nc}{\newcommand}

\nc{\ep}{\varepsilon}
\nc{\iu}{\sqrt{-1}} 
\nc{\integer}{\mathbb{Z}}
\nc{\HBD}{M}
\nc{\Siegel}{D}
\nc{\Hil}{\mathcal{L}}
\nc{\gs}{\mathfrak{s}}

\begin{document}

\title[ Berezin quantization of  homogeneous bounded domains]
{Berezin quantization of  homogeneous bounded domains}
\author[A. Loi, R. Mossa]{Andrea Loi,  Roberto Mossa}
\address{Dipartimento di Matematica e Informatica, Universit\`{a} di Cagliari,
Via Ospedale 72, 09124 Cagliari, Italy}
\email{loi@unica.it;   roberto.mossa@gmail.com }
\thanks{Research partially supported by GNSAGA (INdAM) and MIUR of Italy;
the first author  thank  ESF for short visit grants within the program ``Contact and Symplectic Topology".}
\date{\today}
\subjclass[2000]{53D05; 53C55; 58F06}
\keywords{K\"{a}hler metrics; Berezin quantization; bounded homogeneous domain; Calabi's diastasis function}.

\begin{abstract}
We prove that a homogeneous bounded domain admits a Berezin quantization.
\end{abstract}

\maketitle

\section{Introduction}
Let $\Omega$ be a non-compact complex manifold of dimension $n$, $\omega=\frac{i}{2}\partial\bar\partial\Phi$   a real analytic  \K\  form on 
$\Omega$,  $g$ the corresponding \K\ metric and
$K_{\lambda}(x, y)$ ($\lambda$ a positive constant)  the reproducing kernel for the Bergman space ${\mathcal H}_{\lambda}$ of all holomorphic  functions on $\Omega$ square-integrable against the measure $e^{-\lambda\Phi}\frac{\omega^n}{n!}$. Under the following conditions:
\begin{enumerate}
\item
for all suffciently large real number $\lambda$ there exists a positive constant $c_{\lambda}$ such that  $K_{\lambda}(x, \bar x)=c_{\lambda}e^{\lambda\Phi(x)}$,
\item
the function 
\begin{equation}
e^{-\Phi(x,\ov x) - \Phi(y,\ov y) + \Phi(x,\ov y) + \Phi(y,\ov x)}\nonumber
\end{equation}
is globally defined on $\Omega\times\Omega$, $\leq 1$  and equals $1$ if and only if $x=y$,
(where  $\Phi(x,\ov y)$ be a   sesquianalytic extension  on a neighbourhood  of the diagonal in $\Omega \times \Omega$
of $\Phi$),
 \end{enumerate}
F. A. Berezin \cite{berezin} was able to establish a quantization procedure on ($\Omega, \omega)$. 
His seminal paper  has inspired several  interesting papers both from the mathematical and physical point of view 
(see, e.g.,  M. Engli\v{s}' work \cite{englis},  for a beautiful extension of Berezin's method to complex  domains which satisfiy condition (1) 
only asympotically and 
\cite{cgr1},  \cite{cgr2},  
\cite{cgr3}, \cite{cgr4} for a quantum geometric interpretation of Berezin quantization  and its extension to the compact case (cfr. also
the final  remark  at the end of the paper)).

The only known instances when the above conditions (1) and (2)  are satisfied, however, are just $\Omega = \C^n$ and $\Omega$ a bounded symmetric domain (with the euclidean and the Bergman metric, respectively). 

%In this paper, we extend the quantization procedure to the case when the above condition is satisfied only asymptotically, in an appropriate sense, as ? ? +�. This makes the procedure applicable to a wide class of complex K?ahler manifolds, including all planar domains with the Poincar?e metric (if the domain is of hyperbolic type) or the euclidean metric (in the remaining cases) and some pseudoconvex domains in Cn. Along the way, we also fix two gaps in Berezin�s original paper, and discuss, for � a domain in Cn, a variant of the quantization which uses weighted Bergman spaces with respect to the Lebesgue measure instead of the Ka?hler-Liouville measure |?n|.

\vskip 0.3cm
In this paper we show that these conditions are satisfied by {\em any } homogeneous bounded domain $\Omega\subset \C^n$ equipped with a homogeneous \K\ form $\omega$ (not necessarily the Bergman one).  Our main result is then the  following:

\begin{thm}\label{mainteor}
A bounded homogeneous domain $(\Omega, \omega)$ admits a Berezin quantization.
\end{thm}

Note that condition (1)  can be expressed by saying that Rawnsley's function 
$\epsilon_{\lambda g}(x)=e^{-\lambda\Phi (x)}K_{\lambda}(x, \bar x)$ (see \cite{cgr1})  is a positive constant for all  $\lambda$ sufficiently large
or,  in  a more recent terminology (due to S. Donaldson (\cite{donaldson}) for algebraic manifolds and to the first author and C. Arezzo 
(\cite{arlcomm}) in the noncompact case), that  the   \K\ metric  $g$ is  {\em balanced} for  all $\lambda$
large enough. Moreover,  the function 
 \begin{equation*}
\quad D_g(x,y):=\Phi(x,\ov x) + \Phi(y,\ov y) - \Phi(x,\ov y) - \Phi(y,\ov x)
\end{equation*}
appearing in condition (2)  is the so called  {\em diastasis} function introduced and studied   by  E. Calabi in his seminal paper (\cite{calabi}).

\vskip 0.3cm
The paper  consists of other two sections. 
In the next one we deal with assumption (1) for homogeneous bounded domain, namely  we  study the balanced condition for these domains. More precisely,  we explicitly compute   a real number $\lambda_0$ such that 
$\epsilon_{\lambda g} =c_{\lambda}$ ($c_{\lambda}$ a positive constant)
 iff  $\lambda\geq\lambda_0$.
In Section \ref{quantization}, after briefly  recalling   Berezin's quantization procedure, the definition and the main properties of   Calabi's diastasis function,
we show that the \K\ metric $g$ of a homogeneous bounded domain $\Omega$  satisfies condition (2).
The key tool here,  is the link between  Rawnsley's epsilon function,     Calabi's diastasis function and their  relationship
with  the theory of  \K\ immersions  (following the ideas developed in \cite{cgr1}, \cite{cgr2}, \cite{cgr3} and \cite{cgr4}). Combining the results   of Section \ref{quantization}  with those of Section \ref{balsec} we then  obtain  the proof of Theorem \ref{mainteor}.

\vskip 0.5cm

\noindent {\bf Acknowledgments}. The authors are indebted to
H. Ishi  for the proof of Theorem \ref{thm:balanced} 
describing the structure of balanced  metrics on a homogeneous bounded domain. 

\section{Balanced metrics for homogeneous bounded domains}\label{balsec}
Let $\Omega\subset\C^n$ be a complex domain of $\C^n$  endowed with a K\"ahler metric $g$
and let $\omega$ be the \K\ form associated to $g$, i.e. $\omega (\cdot ,\cdot  )=g(\cdot, J\cdot)$.
Assume that  the metric $g$ can be described by a
strictly plurisubharmonic real valued function $\Phi :\Omega\rightarrow\R$, called a {\em K\"ahler potential} for $g$,
i.e.
$\omega=\frac{i}{2}\de\bar\de \Phi$.

A K\"ahler potential is not unique, in fact it is defined up to an addition with the real part of a holomorphic function on $\Omega$.
Let $\hilb_\Phi$ be the weighted Hilbert space of square integrable holomorphic functions on $\Omega$, with weight $e^{-\Phi}$, namely
\begin{equation}\label{hilbertspacePhi}
\hilb_\Phi=\left\{ s\in\ol(\Omega) \ | \ \, \int_\Omega e^{-\Phi}|s|^2\frac{\omega^n}{n!}<\infty\right\},
\end{equation}
%where $\frac{\omega^n}{n!}=\det(\de\bar \de \Phi)\frac{\omega_0^n}{n!}$ is the volume form associated to $\omega$ and $\omega_0=\frac{i}{2}\sum_{\alpha=0}^{n-1} dz_\alpha\wedge d\bar z_\alpha$ is the standard K\"ahler form on $\C^n$.
If $\hilb_\Phi\neq \{0\}$ we can pick an orthonormal basis $\{s_j\}$ and define its reproducing kernel by
$$K_{\Phi}(z, z)=\sum_{j=0}^\infty |s_j(z)|^2 .$$
Consider the  function
\begin{equation}\label{epsilon}
\varepsilon_g(z)=e^{-\Phi(z)}K_{\Phi}(z, z).
\end{equation}
As suggested by the notation $\varepsilon _g$ depends only on the metric $g$ and not on the choice of the K\"ahler potential $\Phi$. In fact, if $\Phi'=\Phi-\re(\varphi)$, for some holomorphic function $\varphi$, is another potential for $\omega$, we have $e^{-\Phi'}=e^{-\Phi}|e^{\varphi}|^2$. Furthermore, since $\varphi$ is holomorphic and $\de\bar\de \Phi'=\de\bar\de \Phi$,  $e^-\varphi$ is an isomorphism between $\hilb_\Phi$ and $\hilb_{\Phi'}$, and thus we can write $K_{\Phi'}(z, z)=|e^{-\varphi}|^2K_{\Phi}(z, z) $, where $K_{\Phi}(z, z)$ (resp. $K_{\Phi'}(z, z)$) is the reproducing kernel of $\hilb_\Phi$ (resp. $\hilb_{\Phi'}$). It follows that $e^{-\Phi(z)}K_{\Phi}(z, z)=e^{-\Phi'(z)}K_{\Phi'}(z, z)$, as claimed.

\vskip 0,3cm

\noindent
{\bf Definition.}
{\em Let $g$ be a  \K\  metric  on a  complex domain $\Omega\subset\C^n$  such that $\omega =\frac{i}{2}\partial\bar\partial\Phi$. The metric $g$ is \emph{balanced} if  the function $\varepsilon_g$ is a positive  constant.}

In the literature the function $\varepsilon_g$ was first introduced under the name of $\eta$-{\em function} by J.  H. Rawnsley in \cite{rawnsley}, later renamed as $\varepsilon$-{\em function} in \cite{cgr1}. It  also appears under the name of {\em distortion function} for the study of Abelian varieties by  J. R. Kempf \cite{kempf} and S. Ji \cite{ji}, and for complex projective varieties by S. Zhang \cite{zhang} (see also  \cite{gramchev} and references therein). The definition of balanced metrics was given in the compact case by Donaldson  and in  the noncompact case 
by the first author together with C. Arezzo in \cite{arlcomm} (see also \cite{arlquant}).

In this section  we are interested in studying the balanced metrics  when   $(\Omega, g)$ is a homogeneous bounded domain, namely 
$\Omega \subset \C^n$ is   a bounded domain, $g$ a  K\"ahler metric on $\Omega$ and the Lie  group $G=\Aut (\Omega) \cap \Isom  (\Omega,g)$ acts trasitively on $\Omega$, 
where $\Aut (\Omega)$ (resp. $\Isom (\Omega, g)$) denotes the group of invertible holomorphic maps (resp. $g$-isometries) of $\Omega$.
In this case, it is well-known that $\Omega$ is contractible and hence $\omega=\frac{i}{2}\de \deb \Phi$,
for a globally defined \K\ potential $\Phi$.
We start with a lemma.
\begin{lem}\label{lembal}
Let $(\Omega, g, \omega)$ be a bounded homogeneous domain,
$\lambda$  a positive real number and  
\begin{equation}\label{hilbertspace}
\hilb_{\lambda}:=\hilb_{\lambda\Phi}=\left\{ s\in\ol(\Omega) \ | \ \, \int_\Omega e^{-\lambda\Phi}|s|^2\frac{\omega^n}{n!}<\infty\right\}.
\end{equation}
 Then   $\hilb_\lambda\neq \{0\}$ if and only if $\lambda g$ is a balanced metric.
\end{lem}
\begin{proof}
 If $\hilb_\lambda =\{0\}$  then  $\ep_{\lambda g}$ equals the constant zero.
If  $\hilb_\lambda\neq \{0\}$  then the reproducing kernel $K_{\lambda\Phi}$ is not trivial and therefore $\ep_{\lambda g}$ have to be positive at same point. For every $\g\in G$ we have that $\lambda \widetilde \Phi(z) =\lambda \Phi \circ \g$ is a \K\ potential for $\lambda g$ and $K_{\lambda\tilde\Phi}(z,\overline z)=K_{\lambda\Phi}(\g\cdot z,\overline{ \g\cdot z})$ is the reproducing kernel of $\hilb_{\lambda}$. 
Therefore
\begin{equation*}
\begin{split}
\ep_{\lambda g}(z)&=e^{-\lambda\phi(z)} K_{\lambda\Phi}(z,\overline z)=e^{-\lambda\widetilde\phi(z)} K_{\lambda\tilde\Phi}(z,\overline z)\\
&=e^{-\lambda\phi(\g\cdot z)} K_{\lambda\Phi}(\g\cdot z,\overline{ \g\cdot z})=\ep_{\lambda g}(\g\cdot z),
\end{split}
\end{equation*}
since $G$ acts transitively on $\Omega$, we conclude that  $\ep_{\lambda g}$ is a positive constant.
\end{proof}

We are now in the position to state and proof the main result of this section.
\begin{thm} \label{thm:balanced}
Let $(\Omega, g)$ be a homogeneous bounded domain. Then there exists a positive constant $\lambda_0$
such that $\lambda g$ is balanced if and only if $\lambda\geq\lambda_0$.
\end{thm}
Thanks to Lemma \ref{lembal} we are reduced to show Theorem \ref{thm:hilbert} below. This theorem is known to harmonic analysis  specialists and it is essentially  contained in \cite{RV73}.
We present here a (unpublished) proof of this result  due to  H. Ishi  who kindly gave to the authors  the possibility to insert it in this paper.
\begin{thm} \label{thm:hilbert}
There exists $\lambda_0$ such that the Hilbert space 
$\hilb_{\lambda}$ is not trivial if and only if $\lambda \geq  \lambda_0$.
\end{thm}

\begin{proof}
By \cite[Theorem 2 (c)]{D85},
 there exists a connected split solvable Lie subgroup
 $S \subset G$
 acting simply transitively on the domain $\Omega$.
We shall then  reduce the argument to harmonic analysis on the solvable Lie group $S$,
 and apply the results of \cite{RV73}.
 As first step we shall find a specific K\"ahler potential $\Phi$ of $\omega$
 following Dorfmeister \cite{D85}
 (see also \cite[Proof of Theorem 4]{discala}).

Taking a reference point $p_0 \in \Omega$,
 we have a diffeomorphism
\begin{equation} \label{eqn:def_of_iota}
 \iota : S \owns \g \overset{\sim}{\mapsto} \g \cdot p_0 \in \Omega,
\end{equation}
 and
 by  differentiation,
 we get the linear isomorphism
 $\gs := \mathrm{Lie}(S) \owns X
  \overset{\sim}{\mapsto} X \cdot p_0 \in T_{p_0}\HBD \equiv \C^n$.
Then the evaluation of the K\"ahler form
 $\omega$ on $T_{p_o}\Omega$
 is given by
\begin{equation} \label{eqn:beta}
 \omega(X\cdot p_0, Y \cdot p_0) = \beta([X,Y])\qquad
 (X, Y \in \gs)
\end{equation}
 with a certain linear form $\beta \in \gs^*$.
Let $j : \gs \to \gs$
 be the linear map defined in such a way that
 $(jX) \cdot p_0 = \sqrt{-1} (X \cdot p_0)$ for $X \in \gs$. 
We have
 $g(X \cdot p_0,\,Y \cdot p_0) = \beta([jX, Y])$ 
 for $X, Y \in \gs$,
 and the right-hand side defines a positive inner product on $\gs$.
Let $\mathfrak{a}$ be the orthogonal complement of $[\gs, \gs]$ in $\gs$
 with respect to this inner product.
Then $\mathfrak{a}$ is a commutative Cartan subalgebra of $\gs$.
Define $\gamma  \in \mathfrak{a}^*$ by
 \begin{equation} \label{eqn:def_of_gamma}
 \gamma(X) := -4 \beta (jX)
 \end{equation}
 for $X \in \mathfrak{a}$,
 and we extended $\gamma$ to $\gs = \mathfrak{a} \oplus [\gs, \gs]$
 by the zero-extension.
Keeping (\ref{eqn:def_of_iota}) in mind,
 we define a positive smooth function $\Psi$ on $\Omega$ by
\begin{equation} \label{eqn:def_of_Psi}
 \Psi((\exp X) \cdot p_0) = e^{-\gamma(X)} \quad (X \in \gs).
\end{equation}
From the argument in \cite[pp. 302--304]{D85},
 we see that
\begin{equation} \label{eqn:globalpotential}
 \omega = \frac{i}{2}\partial\bar{\partial}\log \Psi.
\end{equation}
It follows that
 $\Phi := \log \Psi$ is a K\"ahler potential of $\omega$.

Noting that $\gamma([\gs, \gs]) = 0$ by definition.
We define the one-dimensional representation
 $\chi : S \to \R^+$
 of $S$ by
 $\chi(\exp X) := e^{-\gamma(X)/2}\,\,\,(X \in \gs).$
Then (\ref{eqn:def_of_Psi}) is rewritten as
 \begin{equation} \label{eqn:Phi-chi}
 e^{-\Phi(\g \cdot p_0)} = \chi(\g)^{-2} \qquad (\g \in S),
 \end{equation}
 so that we have
 $e^{-\Phi(\g \cdot p)} = \chi(\g)^{-2} e^{-\Phi(p)} \,\,\, (p \in \HBD,\,\,\g \in S)$.
%
%
%                                   3
%
%
Let $L^2(S)$ be the Hilbert space of square integrable functions on $S$
 with respect to the left Haar measure $dg$
 for which $\iota_*dg$ equals the $S$-invariant measure $\frac{\omega^n}{n!}$ on $\Omega$.
Then we have the isometry
 $\iota^{\sharp} : {\mathcal H}_{\Phi} \to L^2(S)$
 defined by
 \begin{equation} \label{eqn:def_of_iotasharp}
 \iota^{\sharp} F(\g)
 := e^{- \Phi\circ \iota (\g)/2} F \circ \iota(\g)
  = e^{- \Phi(\g \cdot p_0)/2} F(\g \cdot p_0),
 \end{equation}
for  $\g \in S$ and $F \in {\mathcal H}_{\Phi}$.
We shall give an analytic description of
 the image of $\iota^{\sharp}$ in $L^2(S)$.
For $X \in \gs$,
 we denote by $R(X)$
 the corresponding left invariant vector field
  on the Lie group $S$.
Namely,
 for $\varphi \in C^{\infty}(S)$ we have
 $R(X)\varphi(\g) := \frac{d}{dt}|_{t=0} \varphi(\g \exp t X)$, $(\g \in S)$.
For $Z = X + \iu Y \in \gs_{\C}$ with $X, Y \in \gs$,
 we define $R(Z) := R(X) + \iu R(Y)$.
Let $\gs_-$ be the subspace
 $\set{X + \iu j X}{X \in \gs}$
 of $\gs_{\C}$.
Then we have a linear isomorphism $\gs_- \owns Z \mapsto Z \cdot p_0 \in T_{p_0}^{(0,1)}\Omega$,
 so that the push-forward $\iota_*R(Z)$ is
 an $S$-invariant anti-holomorphic vector field
 on the complex domain $\Omega$
 for each $Z \in \gs_-$.
Thus $\gs_-$ is a complex Lie subalgebra of $\gs_{\C}$.

Let $f$ denote the linear form $-2\beta$ on $\gs$. We now pause to prove the following proposition.

%%%%%%%%%%%%%%%%%%%%%%%%%%%%%%%%%%%%%%%%%%%%%%%%%%%%
%             Prop (polarization)
%%%%%%%%%%%%%%%%%%%%%%%%%%%%%%%%%%%%%%%%%%%%%%%%%%%%

\begin{prop} \label{prop:polarization}
{\rm (i)}
The subalgebra $\gs_- \subset \gs_{\C}$ is a positive polarization at $f$,
 that is,
 $f([\gs_-, \gs_-]) = 0$
 and $\iu f([Z, \bar{Z}]) \ge 0$ for all $Z \in \gs_-$.\\
{\rm (ii)}
The image of $\iota^{\sharp}$
 equals the function space
 \footnote{The function space $\Hil(S,f)$ is fundamental in theory of holomorphic induction,
 which is closely related to Kostant's geometric quantization \cite{AK71}.}
\[
 \Hil(S,f) := \set{\varphi \in L^2(S)}{R(Z)\varphi = - \iu f(Z) \varphi \mbox{ for all }Z \in \gs_-}.
\]
\end{prop}
\noindent
\pf 
(i) For $Z, Z' \in \gs_-$,
 we have by \eqref{eqn:beta}
\[
 f([Z,Z']) = -2\omega(Z \cdot p_0, Z' \cdot p_0) = 0
\]
 because $\omega$ is $(1,1)$-form.
Similarly,
 we have for $Z = X + \iu jX \in \gs_-$ with $X \in \gs$,
 \[
 \iu f([Z, \bar{Z}]) = 2 \iu \omega(\bar{Z} \cdot p_0, Z \cdot p_0) = 4 g(X \cdot p_0, X\cdot p_0)/2 \ge 0.
\]
(ii)
We take $\varphi = \iota^{\sharp}F \in \mathrm{Image}\,\iota^{\sharp}$
 with $F \in L^2_{hol}(\HBD, e^{-\Phi}\frac{\omega^n}{n!})$.
By (\ref{eqn:Phi-chi}) and (\ref{eqn:def_of_iotasharp}),
 we obtain
\[
 \varphi(\g) = \chi(\g)^{-1} F(\iota(\g)), \ (\g \in G).
\]
By the Leibniz rule,
 we have for $Z = X + \iu j X \in \gs_-\,\,(X \in \gs)$
\[
 R(Z)\varphi = R(Z)\chi^{-1} \cdot F \circ \iota + \chi^{-1} \cdot R(Z) F \circ \iota.
\]
Since $F$ is holomorphic,
 we have
 $R(Z) F \circ \iota = \iota_*R(Z)F = 0$.
Noting that
 $\chi^{-1/2}$ is a one-dimensional representation,
 we have
 \[
 R(X)\chi^{-1} = \frac{\gamma(X)}{2} \chi^{-1} = f(jX) \chi^{-1}
 \]
 because
 \begin{equation} \label{eqn:gamma-f}
 \gamma(X) = - 4\beta(jX) = 2 f(jX).
 \end{equation}
Indeed,
 we may assume that
 (\ref{eqn:def_of_gamma}) holds for all $X \in \gs_-$
 (\cite[Page 362]{RV73}).
Then $R(jX) \chi^{-1} = -f(X) \chi^{-1}$,
 so that $R(Z) \chi^{-1} = - \iu f(Z) \chi^{-1}$.
Thus we have $\mathrm{Image}\, \iota^{\sharp} \subset \Hil(S,f)$.
The converse inculsion can be shown similarly.
\qed
$ $\\
%%%%%%%%%%%%%%%%%%%%%%%%%%%%%%%%%%%%%%%%%%%%%%%%%%%%
%%%%%%%%%%%%%%%%%%%%%%%%%%%%%%%%%%%%%%%%%%%%%%%%%%%%
%%%%%%%%%%%%%%%%%%%%%%%%%%%%%%%%%%%%%%%%%%%%%%%%%%%%

By Proposition ~\ref{prop:polarization} (ii),
 the non-vanishing condition of $\hilb_{\Phi}$ is equivalent to the one of $\Hil(S,f)$,
 and the latter is completely determined by  \cite[Theorem 4.26]{RV73}.
In order to apply the results in \cite{RV73},
 we need a root structure of the Lie algebra $\gs$ with respect to $\mathfrak{a}$
 due to Piatetskii-Shapiro \cite{PS69}.
For a linear form $\alpha$ on the Cartan algebra $\mathfrak{a}$,
 we denote by $\gs_{\alpha}$ the subspace
 $\set{X \in \gs}{[C,X] = \alpha(C)X \,\, (\forall C \in \mathfrak{a})}$
 of $\gs$.
We say that $\alpha$ is a \textit{root}
 if $\gs_{\alpha} \ne \{0\}$ and $\alpha \ne 0$.
Thanks to \cite[Chapter 2, Section 3]{PS69} or \cite[Theorem 4.3]{RV73},
 there exists a basis
 $\{\alpha_1, \dots, \alpha_r\}$, $(r := \dim \mathfrak{a})$
 of $\mathfrak{a}^*$
 such that every root is one of the following:
\[
 \alpha_k, \ \alpha_k/2, \ (k=1, \dots, r), \qquad (\alpha_l \pm \alpha_k)/2, \ (1 \le k < l \le r).
\]
If $\{A_1, \dots, A_r\}$ is
 the basis of $\mathfrak{a}$ dual to $\{\alpha_1, \dots, \alpha_r\}$,
 then $\gs_{\alpha_k} = \R jA_k$.
Thus $\gs_{\alpha_k}\,\,(k=1, \dots, r)$ is always one dimensional,
 whereas other root spaces $\gs_{\alpha_k/2}$ and $\gs_{(\alpha_l \pm \alpha_k)/2}$ may be $\{0\}$.
We put for $k=1, \dots, r$
 \[
 p_k := \sum_{i<k} \dim \gs_{(\alpha_k - \alpha_i)/2}, \
 q_k := \sum_{l>k} \dim \gs_{(\alpha_l - \alpha_k)/2}, \
 b_k := \frac{1}{2}\dim \gs_{\alpha_k/2},
 \]
 see \cite[Definition 4.7]{RV73} and \cite[(2.7)]{N01}.
%Note that $b_k$ as well as $p_k$ and $q_k$ is integer because the root space $\gs_{\alpha_k}$ is of even dimensional.

Since $\{\alpha_1, \dots, \alpha_r\}$ is a basis of $\mathfrak{a}^*$,
 the linear form $\gamma$ is written as
 $\gamma = \sum_{k=1}^r \gamma_k \alpha_k$
 with unique $\gamma_1, \dots, \gamma_r \in \R$,
 where $\alpha_k$'s are extended to $\gs$ by the zero-extension.
Since $j A_k \in \gs_{\alpha_k}$,
 we obtain

\begin{equation*}
\begin{split}
 \gamma_k &= \gamma(A_k) = -4 \beta (jA_k) = -4 \beta([A_k, jA_k])= 4 \beta([jA_k, A_k])\\ &
  = 4 g (A_k \cdot p_0, A_k \cdot p_0) >0.
\end{split}
\end{equation*}
By (\ref{eqn:gamma-f}),
 we have $f = - \gamma \circ j /2 = \sum_{k=1}^r (- \gamma_k/2) \alpha_k \circ j$.
Now we get the following from \cite{RV73}.

%%%%%%%%%%%%%%%%%%%%%%%%%%%%%%%%%%%%%%%%%%%%%%%%%%%%
%             thm (Rossi-Vergne)
%%%%%%%%%%%%%%%%%%%%%%%%%%%%%%%%%%%%%%%%%%%%%%%%%%%%

\begin{prop}[\mbox{\cite[Theorem 4.26]{RV73}}] \label{thm:Rossi-Vergne}
The Hilbert space $\Hil(S,f)\neq 0$  if and only if
\[
 \gamma_k > 1 + p_k + b_k + q_k/2,\ (k=1, \dots, r).
\]
\end{prop}

\vskip 0,5cm

The proof  of Theorem \ref{thm:hilbert}  follows now easily by  the   proof of Proposition \ref{thm:Rossi-Vergne}
by replacing the Hilbert space ${\mathcal H}_{\Phi}$ by 
${\mathcal H}_{\lambda}={\mathcal H}_{\lambda\Phi}$
and by setting
\begin{equation} \label{eqn:def_of_lambda0}
 \lambda_0 := \max_{1 \le k \le r} \frac{1 + p_k + b_k + q_k/2}{\gamma_k}.
\end{equation}
\end{proof}

\begin{exmp}\rm
If $\omega$ is the K\"ahler form corresponding to the Bergman metric on $\Omega$,
 then $\gamma_k = 2 + p_k + q_k + b_k\,\,\,(k=1, \dots, r)$
 (see \cite[Theorem 5.1]{G64} or \cite[(2.19)]{N01}),
 so that
 \[
 \lambda_0 = \max_{1 \le k \le r} \frac{1 + p_k + b_k + q_k/2}{2 + p_k + q_k + b_k},
 \]
 which is found in \cite[Page 97]{N01}.
In particular,
 if $\Omega$ is a bounded symmetric domain, then there exists integers $a$ and $b$
 so that
 \[
 p_k = (k-1)a, \quad
 q_k = (r-k)a, \quad
 b_k = b, \quad
 \gamma_k = (r-1)a + b + 2.
 \]
Therefore 
\begin{equation*}
\begin{split}
 \lambda_0 &= \max_{1 \le k \le r} \frac{1 + (k-1)a + b + (r-k)a/2}{\gamma}
 = \frac{1 + (r-1)a + b}{\gamma}\\& = \frac{\gamma - 1}{\gamma},
\end{split}
\end{equation*}
which is consistent with \cite[Theorem 1]{LZ10}.
\end{exmp}

\begin{rmk}\rm
Since  a balanced metric is projectively induced (see e.g. \cite{cgr1} and  \cite{LZ10} or the proof of Theorem  \ref{mainteor} in the next  section), 
it is natural to ask for which $\lambda \in \R ^+$ the metric $\lambda g$ is projectively induced.
This problem was addressed and solved in
\cite[Theorem 2]{articwall}
for bounded symmetric domains and in  \cite[Theorem 4]{discala}
for the more general case of  homogeneous bounded domains considered in the present paper.  For completeness we briefly  recal  here the results obtained in \cite{discala}.
The crucial point is that the homogeneous  metric $\lambda g$ on a homogeneous bounded domain  is projectively induced if and only if the analytic extension  $e^{\lambda \Phi (z, \overline w)}$ of $e^{\lambda \Phi(z)}$ is the reproducing kernel for an Hilbert space. The condition for $e^{\lambda \Phi(z,\overline w)}$ to be a reproducing kernel is described in \cite{ishi1}.  
Then Theorem 15 in \cite{ishi2} tells us that
\begin{equation*}
%\begin{split}
\{0,c_0\}\cup(c_0, + \infty) \subset  W(g) \cup \{0\} \subset \set{\frac{q_k}{2\gamma_k}}{k=1, \dots, r} \cup (c_0, +\infty),
%\end{split}
\end{equation*}
where $c_0 := \max \set{\frac{q_k}{2\gamma_k}}{k=1, \dots, r}$. Thefore $W(g)$ consists of a continuous part and of a discrete part with at most $r$ elements. It is interesting to note that in general the constant $c_0$ is different from the constant $\lmb_0$  in  Theorem \ref{thm:balanced}. This implies for example that,  on any homogeneous bounded domain,   there exist infinite (homothetic) projectively induced metrics which are not balanced.
\end{rmk}

\section{Berezin quantization and the proof of Theorem \ref{mainteor}}\label{quantization}
Let $(\Omega, \omega)$ be a symplectic manifold and let $\{\cdot, \cdot \}$ be the associated Poisson bracket.
A {\em Berezin quantization} (we refer to \cite{berezin} for details) on $\Omega$ is given by a family of associative algebras $\mathcal A_\hbar$
where the parameter $\hbar$ (which plays the role of Planck constant) range over a set $E$ of positive reals with limit point $0$.
Then in the direct sum $\oplus_{h \in E} \A_h$ with component-wise product $*$, one chooses a subalgebra ${\mathcal A}$, such that for an arbitrary element 
$f=f(\hbar)\in {\mathcal A}$, where $f(\hbar)\in {\mathcal A}_\hbar$, there exists a limit $\lim_{\hbar\rightarrow 0} f(\hbar)=\varphi (f)\in C^{\infty}(\Omega)$.
The following correspondence principle must  hold: for $f, g \in {\mathcal A}$
$$\varphi (f*g)=\varphi (f)\varphi (g),\    \    \  \varphi\left(\hbar^{-1}(f*g-g*f)\right)=i\{\varphi(f), \varphi (g)\}.$$
Moreover, for any pair of points $x_1,x_2 \in \Omega$ there exists $f \in \A $ such that $\varphi (f) (x_1) \neq \varphi(f)(x_2)$.
\vskip 0.3cm
Consider now a real analytic \K\ manifold $\Omega$, with \K\ metric $g$ and associated \K\ form $\omega$. Assume that there exists a (real analytic) {global \K\ potential} $\Phi :\Omega\rightarrow\R$.
This function  extends to a sesquianalytic function $\Phi(x,\ov y)$ on a neighbourhood  of the diagonal in $\Omega \times \Omega$ such that $\Phi(x,\ov x)=\Phi(x)$.  Consider  {\em Calabi's diastasis function} $D_g$ defined on a neighbourhood 
of the diagonal in $\Omega \times \Omega$ by:
\begin{equation*}
%\begin{split}
\qquad D_g(x,y)=\Phi(x,\ov x) + \Phi(y,\ov y) - \Phi(x,\ov y) - \Phi(y,\ov x).
%\end{split}
\end{equation*}
By its definition we see that  Calabi's  diastasis function   is independent from the potential chosen which is defined up to the sum with the real part of a holomorphic function.
Moreover,  it is easily seen that $D_g$ is real-valued,  symmetric in $x$ and $y$  and $D_g(x, x)=0$
(the reader is referred to \cite{calabi}  and \cite{diastherm} for more details on the  diastasis function).

\begin{exmp}\label{diastproj}\rm
Let $g_{FS}$ be the Fubini--Study metric on the infinite dimensional
complex proective space $\C P^{\infty}$
of holomorphic sectional curvature
$4$ and let
$D_{g_{FS}}(p, q)$  be the associated Calabi's  diastasis function.
One  can show that 
for all $p\in \C P^{\infty}$
the  function $D_{g_{FS}}(p, \cdot)$ 
is globally defined except in the cut locus
$H_p$ of $p$ where it blows up.
Moreover $e^{-D_{g_{FS}}(p, q)}$
is globally defined and smooth on 
$\C P^{\infty}$,
$e^{-D_{g_{FS}}(p, q)}\leq 1$
and $e^{-D_{g_{FS}}(p, q)}= 1$
if and only if $p=q$ (see \cite{diastherm}  for details).
\end{exmp}

 \vskip 0.3cm
The following theorem is a reformulation  of Berezin   quantization result    (see Engli\v{s} \cite{englis})    in terms of Rawnsley $\ep$-function   and  Calabi's diastasis function.

\begin{thm}\label{thmberezin}
Let $\Omega\subset \C^n$ be a complex domain equipped with a real analytic  \K\ form $\omega$ and corresponding \K\ metric $g$.
Then, $(\Omega, \omega)$ admits a Berezin quantization if the following two  conditions are satisfied:
\begin{enumerate}
\item
Rawnsley's function $\ep_{\lambda g}(x)$ is a positive constant for all sufficiently large $\lambda$;
\item
the function $e^{-D_g(x, y)}$  is globally defined on $\Omega\times\Omega$,  $e^{-D_g(x, y)}\leq1$  and $e^{-D_g(x, y)}=1$,  if and only if $x=y$,
\end{enumerate}
\end{thm}

\vskip 0.3cm

We are now in the position to prove  Theorem \ref{mainteor}, namely that a bounded homogeneous domain admits a Berezin quantization.

\begin{proof}[Proof of Theorem \ref{mainteor}]
Condition  (1) in Theorem \ref{thmberezin}  is a consequence of  Theorem  \ref{thm:balanced} of  the previous section.
Hence,   it remains  to show that  (2) in  Theorem \ref{thmberezin} is satisfied by a homogeneous bounded domain $(\Omega, \omega, g)$
\footnote{It is well-known that $g$ is real analytic so it makes sense to verify (2).}.
Fix $\lambda\geq  \lambda_0$ with $\lambda_0$ given by Theorem \ref{thm:balanced}. 
Consider the  {\em coherent states map} (see \cite{rawnsley}), namely the holomorphic map $\varphi_{\lambda}:\Omega\rightarrow\C P^{\infty}$ from $\Omega$ into the infinite dimensional
complex projective space $\C P^{\infty}$ defined by
$$\varphi_{\lambda}: \Omega\f \C P^{\infty},\  x\mapsto [s_0^{\lambda}(x), \dots ,s_j^{\lambda}(x), \dots ],$$
where $\{s_j^{\lambda}\}_{j=0, \dots }$ is an orthonormal basis of the Hilbert space $\hilb_{\lambda}$ given by (\ref{hilbertspace}).
This map is well-defined since $\epsilon_{\lambda g}$ is a positive constant and hence, 
for a given $x\in\Omega$,  there exists $j_0$ such that $s_{j_0}$ does not vanish at $x$.
Moreover, the constancy of $\ep_{\lambda g}$ implies that 
$\varphi_{\lambda}^*g_{FS}=\lambda g$,
where $g_{FS}$ is the Fubini--Study  metric on $\C P^{\infty}$ (see \cite{rawnsley} for a proof).
In  other words,  the metric $\lambda g$ is projectively induced via
the  coherent states map. By Example \ref{diastproj},
Calabi's diastasis function $D_{g_{FS}}$ of $\C P^{\infty}$ is such that  $e^{-D_{FS}}$
is globally defined on $\C P^{\infty}\times\C P^{\infty}$ and by the hereditary property of the diastasis function
 (see \cite[Proposition  6 ]{calabi}) we get that, for all $x, y\in\Omega$,
\begin{equation}\label{fondequ}
e^{-D_{FS}(\varphi_{\lambda}(x), \varphi_{\lambda}(y) )}=e^{-D_{\lambda g}(x, y)}= e^{-\lambda D_{g}(x, y)}=\left(e^{- D_{g}(x, y)}\right)^{\lambda}
\end{equation}
 is globally defined on $\Omega\times\Omega$.
 Since, by Example \ref{diastproj}),   $e^{-D_{FS}(p, q)}\leq 1$ for all $p, q\in \C P^{\infty}$ it follows that $e^{-D_g(x, y)}\leq 1$ for  all
 $x, y\in \Omega$.
It remains to show that   $e^{-D_g(x, y)}=1$ iff $x=y$.  By formula (\ref{fondequ})   
%(cfr.  also \cite[Theorem 12 (d)]{calabi})  
and by the fact that $e^{-D_{FS}(p, q)}=1$ iff $p=q$ (again by Example \ref{diastproj})
this is equivalent to   the injectivity of the coherent states map  $\varphi_{\lambda}$. This, in turns, follows by a recent result  \cite[Theorem 3]{discala} of the first author together with A. J. Di Scala and H. Ishi  which asserts that   a   \K\ immersion  of a homogeneous \K\ manifold (not necessarily a  bounded domain) into a  finite or infinite dimensional complex projective space is  one to one.
\end{proof}

\vskip 0.5cm

\noindent
{\bf Final remark}
All the results of this paper can be formulated in term of geometric quantization tools, i.e. holomorphic line bundle, coherent states, charactersitic $2$-point function
as  in \cite{cgr1}, \cite{cgr2}, \cite{cgr3} and \cite{cgr4}.
For a bounded homogeneous domain $(\Omega, g, \omega )$ the  quantum line bundle, namely the holomorphic line bundle over $\Omega$ whose first Chern class  equals the  De Rham class   of  $\omega$, is trivial and so its global  holomorphic sections can be identified
with the holomorphic functions on $\Omega$.
A natural problem one could tries to solve is that of obtaining  a  quantization by deformation (or equivalently a $*$-product) of a bounded homogeneous domain starting from the Berezin quantization given by  Theorem \ref{mainteor} and  following the ideas developed in \cite{cgr4}  for the case of bounded symmetric domains. This is not a straightforward matter.  Indeed in  \cite{cgr4}   Cahen, Gutt and Rawnsley use the  
polydisk theorem while for  a general homogeneous bounded  domain no such  theorem seems to be  avaliable.  Moreover, to the authors  best knowledge,  a classification of  all bounded homogeneous   domains  is still missing.
We will  attack this problem in a forthcoming paper.
We finally point out   that  in  \cite{cgr2}  Cahen, Gutt and Rawnsley
prove the existence of a Berezin  $*$-product on any  compact flag manifold.
The results obtained  by  Cahen, Gutt and Rawnsley  for the compact case (see in particular  \cite[Proposition 3]{cgr2} which represents
the analogous of our  Theorem \ref{mainteor}  and  Theorem \ref{thm:balanced}) 
strongly rely on Kodaira's theory for  projective algebraic manifolds
which  can not be apply  to the noncompact context.

\end{document}